\documentclass[12pt]{amsart}
\usepackage{amscd, amsthm}
\usepackage{amsmath,amsfonts,amsthm,epsfig,latexsym,graphicx,amssymb,psfrag}
\usepackage{amssymb,amsthm,amsmath,xspace}
\usepackage[english]{babel}
\usepackage{xcolor}
\usepackage{enumerate}
\usepackage{array} 
\usepackage{float}

\newtheorem{theo}{Theorem}
\newtheorem{theor}{Theorem}[section]
\newtheorem{coro}{Corollary}
\newtheorem{prop}{Proposition}[section]
\newtheorem{lemm}[prop]{Lemma}
\newtheorem{propri}[prop]{Properties}

\newtheorem{claim}[prop]{Claim}
\newtheorem{exem}{Examples}[section]

\newtheorem*{rem}{Remark}

\newtheorem*{nota}{Notations}

 \newtheoremstyle{defi}
 {3 pt}{3 pt}{\rm}{3 pt}{\bf}{\hskip -0.5 pt . }{0 pt}
 {\thmname{#1}\thmnumber{#2}\thmnote{\textnormal(#3)}}
 \theoremstyle{defi}

 \newtheoremstyle{rem}
 {3 pt}{3 pt}{\rm}{3 pt}{\bf}{\hskip -0.5 pt . }{0 pt}
 {\thmname{#1}\thmnumber{#2}\thmnote{\textnormal(#3)}}
 \theoremstyle{rem}

 \newtheorem{defi}[prop]{Definition }
  
 \newtheorem{rema}[prop]{Remark }
 
\newcommand{\N}{\mathbb{N}} \newcommand{\Z}{\mathbb{Z}} \newcommand{\Q}{\mathbb{Q}}  

\def \ds {\displaystyle} \def \sct {\scriptstyle}

\newcommand{\iet}{\mathcal G}
\newcommand{\ietd}{{\mathcal G}_2}
\newcommand{\PL}{\textsf{PL}}
\newcommand{\fix}{\textsf{Fix}}
\newcommand{\Per}{\bf \textsf{Per}} 
\newcommand{\CR}{\textsf{C}}
\newcommand{\CRd}{{\bf\textsf{C}}}
\newcommand{\BP}{\textsf{BP}}
\newcommand{\PIET}{\textsf{M}\mathcal G_2} 
\newcommand{\pbul}{\scriptstyle \bullet}
\newcommand{\pal}{\scriptstyle \alpha}
\newcommand{\pp}{\scriptstyle +} \newcommand{\pmi}{\scriptstyle -}
\newcommand{\pl}{P} 
\newcommand{\id}{\textsf{Id}} %
\newcommand{\G}{G} 

\def\ra{\mathop{\ \ \longrightarrow \ \ }} \def\Ra{\mathop{\Longrightarrow}}  \def\nearsubset{\mathop{\subset}}
\def\LRA{\mathop{\Longleftrightarrow}}

\title[Interval Exchange Transformations groups]{\small Interval Exchange Transformations groups \\ Free actions and dynamics of virtually abelian groups} 

\author{ Nancy Guelman and Isabelle Liousse}

\begin{document}

\address{{\bf  Nancy  GUELMAN}, \rm{IMERL, Facultad de Ingenier\'{\i}a, {Universidad de la Rep\'ublica, C.C. 30, Montevideo, Uruguay.} \textit{nguelman@fing.edu.uy}}.}

\address{{\bf Isabelle LIOUSSE}, \rm{Univ. Lille, CNRS, UMR 8524 - Laboratoire Paul Painlevé, F-59000 Lille, France. \textit {isabelle.liousse@univ-lille.fr}}.}
\begin{abstract}

H\"older's theorem states that any group acting freely by circle homeomorphisms is abelian, this is no longer true for interval exchange transformations: we first give examples of free actions of non abelian groups.

Then after noting that finitely generated groups acting freely by IET are virtually abelian, we classify the free actions of groups containing a copy of $\mathbb Z^2$, showing that they are ``conjugate" to actions in some specific subgroups $G_n$, namely $G_n \simeq  (\ietd)^n \rtimes\mathcal S_n $ where $\ietd$ is the group of circular rotations seen as exchanges of $2$ intervals and $\mathcal S_n$ is the group of permutations of $\{1,...,n\}$ acting by permuting the copies of $\ietd$.

We also study non free actions of virtually abelian groups and we obtain the same conclusion for any such group that contains a conjugate  to a product of restricted rotations with disjoint supports and without periodic points.

As a consequence, we provide examples of non virtually nilpotent subgroups of IETs. In particular, we show that the group generated by $f\in G_n$ periodic point free and $g\notin G_n$ is not virtually nilpotent. Moreover, we exhibit examples of finitely generated non virtually nilpotent subgroups of IETs, some of them are metabelian and others are not virtually solvable.
\end{abstract}

\maketitle

\section{Introduction.}

\begin{defi} An {\bf interval exchange transformation (IET)} is a bijective map $f: [0,1) \rightarrow [0,1)$ defined by a finite partition of the unit interval into half-open subintervals and a  reordering of these intervals by translations. If the partition has cardinality $r$, we say that $f$ is an $\mathbf{r}${\bf -IET}. 

\smallskip

Circular rotations identify with 2-IETs and an IET $g$ whose support is a subinterval $I=[a,b)$ is a {\bf restricted rotation} if the orientation preserving affine map from $I$ to $[0,1)$ conjugates $g_{\vert I}$ to a 2-IET. It is denoted by $\boldsymbol{R_{\alpha, I}}$ where $\alpha\in \frac{[a,b)}{b\sim a}$ is represented by $g(a)-a$.

We denote by $\boldsymbol \iet$ the group consisting of all IETs and by $\boldsymbol {\ietd}$ the group consisting of all circular rotations, $\boldsymbol{R_\alpha}$ ($\alpha\in \mathbb S^1$), regarded as 2-IETs.

\smallskip

An {\bf affine interval exchange transformation (AIET)} is a bijective map $f: [0,1) \rightarrow [0,1)$ defined by a finite partition of the unit interval into half-open subintervals such that the restriction to each of these intervals is an orientation preserving affine map.

An AIET that is a homeomorphism of $[0,1)$ is called {\bf PL-homeomorphism}.

\smallskip

A group is said to {\bf act freely} if the only element acting with fixed points is the trivial element, in particular a free action is faithful.
\end{defi}

\smallskip

By H\"older's Theorem (see e.g. Theorem 2.2.32 of \cite{NaB}), groups that act freely by circle homeomorphisms are abelian.

\begin{rema}\label{R1Gro} Since translations commute, the orbit of any point under a finitely generated subgroup $H$ of $\iet$ has polynomial growth. Therefore, if $H$ acts freely then it has polynomial growth so it is virtually nilpotent by {Gromov} \cite{Gro}. In addition, as any finitely generated nilpotent group is polycyclic, Theorem 2 of \cite{DFG} implies that $H$ is virtually abelian.
\end{rema}

In order to show that H\"older's Theorem is no longer true for free actions by elements of $\iet$, we let the reader check that: 

\begin{enumerate}

\item The following IETs $a$ and $b$ generate a free and minimal action of the Baumslag-Solitar group $BS(1,-1):=\langle \ a,b \ \vert \  ba b^{-1}=a^{-1} \ \rangle$ provided that $1$, $\alpha$, $\beta_1$ and $\beta_2$ are rationally linearly independent numbers. 

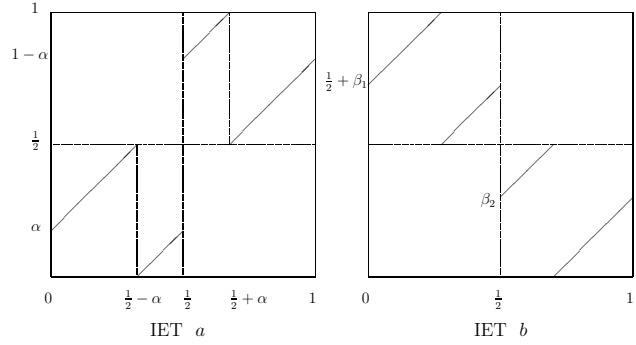
\begin{figure}[H]  
\scalebox{.50}{
\begin{picture}(450,200)
\put(0,0){\line(1,0){200}} \put(0,200){\line(1,0){200}}
\put(0,0){\line(0,1){200}} \put(200,0){\line(0,1){200}}

\put(-5,-20){$0$}
\put(100,-20){${\frac{1}{2}}$}
\put(195,-20){$1$}

\put(-15,100){$\frac{1}{2}$}
\put(-15,200){$1$}

\put(65,0){\dashbox(0,100){}}
\put(135,100){\dashbox(0,100){}}
\put(100,0){\dashbox(0,200){}}

\put(0,100){\dashbox(200,0){}}

\put(0,35){\line(1,1){65}}
\put(65,0){\line(1,1){35}}
\put(100,165){\line(1,1){35}}
\put(135,100){\line(1,1){65}}

\put (-15,35){$\alpha$}
\put (-30,165){$1-\alpha$}
\put(55,-20){${\frac{1}{2}}-\alpha$}
\put (135,-20){${\frac{1}{2}}+\alpha$}

\put (75,-45){\large{IET \ $\displaystyle a$}}

\put(240,0){\line(1,0){200}} \put(240,200){\line(1,0){200}}
\put(240,0){\line(0,1){200}} \put(440,0){\line(0,1){200}}

\put(235,-20){$0$}
\put(335,-20){${\frac{1}{2}}$}
\put(435,-20){$1$}

\put(340,0){\dashbox(0,200){}}

\put(240,100){\dashbox(200,0){}}

\put(240,145){\line(1,1){55}}
\put(295,100){\line(1,1){45}}
\put(340,60){\line(1,1){40}}
\put(380,0){\line(1,1){60}}

\put (207,145){$\frac{1}{2}+\beta_1$}
\put (325,55){$\beta_2$}

\put (320,-45){\large{IET \ $\displaystyle b$}}

\end{picture} }
\end{figure} \ 

\smallskip

Moreover, this example completes the authors result (\cite{GL}) proving that the Baumslag-Solitar groups (\cite{BS}) defined by $BS(m,n)= \langle\ a,b \ \vert \  ba^m b^{-1}=a^n\ \rangle$ do not act faithfully by AIET when $|m| \neq |n|$.

\item The following IETs $a$ and $b$ generate a free and minimal action of the crystallographic group $C_1 =\langle a, b \  |  \ b a^2  b^{-1} = a^{-2},  a b^2 a^{-1} =b^{-2}\rangle$ provided that $1$, $\alpha$ and $\beta$ are rationally linearly independent numbers.

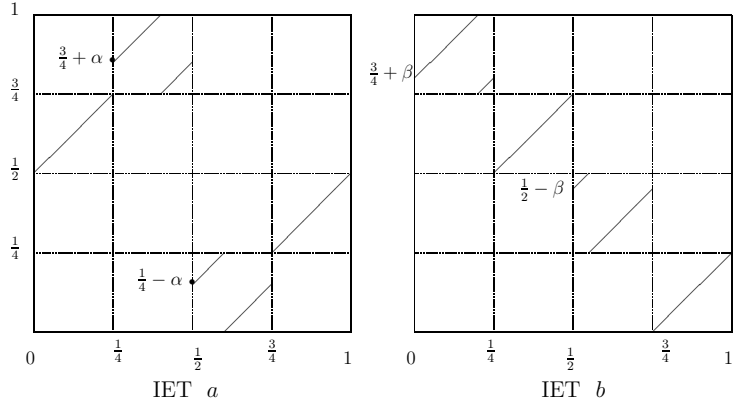
\begin{figure}[H] 
\scalebox{.60}{
\begin{picture}(450,200)
\put(0,0){\line(1,0){200}} \put(0,200){\line(1,0){200}}
\put(0,0){\line(0,1){200}} \put(200,0){\line(0,1){200}}

\put(-5,-20){$0$}
\put( 50,-15){$\frac{1}{4}$}
\put(100,-20){${\frac{1}{2}}$}
\put(147,-15){$\frac{3}{4}$}
\put(195,-20){$1$}

\put(-15,50){$\frac{1}{4}$}
\put(-15,100){$\frac{1}{2}$}
\put(-15,150){$\frac{3}{4}$}
\put(-15,200){$1$}

\put(50,0){\dashbox(0,200){}}
\put(100,0){\dashbox(0,200){}}
\put(150,0){\dashbox(0,200){}}

\put(0,50){\dashbox(200,0){}}
\put(0,100){\dashbox(200,0){}}
\put(0,150){\dashbox(200,0){}}

\put(0,100){\line(1,1){50}}

\put (50,170){\line(1,1){30}}
\put (80,150){\line(1,1){20}}
\put (15,170){$\frac{3}{4}+\alpha \ \pbul$}

\put (100,30){\line(1,1){20}}
\put (120,0){\line(1,1){30}}
\put (65,30){$\frac{1}{4}-\alpha \ \pbul$}

\put (150,50){\line(1,1){50}}
\put ( 75,-41){\large{IET  \ $a$}}

\put(240,0){\line(1,0){200}} \put(240,200){\line(1,0){200}}
\put(240,0){\line(0,1){200}} \put(440,0){\line(0,1){200}}

\put(235,-20){$0$}
\put(285,-15){${\frac{1}{4}}$}
\put(335,-20){${\frac{1}{2}}$}
\put(395,-15){${\frac{3}{4}}$}
\put(435,-20){$1$}

\put(290,0){\dashbox(0,200){}}
\put(340,0){\dashbox(0,200){}}
\put(390,0){\dashbox(0,200){}}

\put(240,50){\dashbox(200,0){}} 
\put(240,100){\dashbox(200,0){}}
\put(240,150){\dashbox(200,0){}}

\put (240,160){\line(1,1){40}}
\put (280,150){\line(1,1){10}}
\put (211,160){$\frac{3}{4}+ \beta$}
\put(290,100){\line(1,1){50}}
\put (340,90){\line(1,1){10}}
\put (350,50){\line(1,1){40}}
\put (306,87){$\frac{1}{2}-\beta$}
\put(390,0){\line(1,1){50}}
\put ( 320,-41){\large{IET  \ $b$}}

\end{picture}} 
\end{figure} \

\end{enumerate}

Note that the two previous groups are subgroups of some very specific subgroups of $\iet$, namely the groups $G_n$ described in the following

\begin{defi} \label{Gn}   
Let $n$ be a positive integer and $\mathbb S_n= [0, \frac{1}{n}]/{\scriptstyle 0=\frac{1}{n}}$ be the circle of length $\frac{1}{n}$.  We define $\boldsymbol{G_n}$ as the set of IETs on $[0,1)$ that preserve the partition $[0,1)=[0,\frac{1}{n})\cup [\frac{1}{n}, \frac{2}{n}) ... \cup [\frac{n-1}{n},1)$ and whose restrictions to the intervals $I_i=[\frac{i-1}{n}, \frac{i}{n})$ are IETs with only one interior discontinuity.

\smallskip

For $g\in G_n$, we define $\boldsymbol{\sigma_g}$ as the element of $\mathcal S_n$ given by $\sigma_g(i)=j$ if $g(I_i) = I_j$. 

It follows that $g_{\vert I_i} = R_{\alpha_i,\sigma_g(i)}$ where $R_{\alpha_i,\sigma_g(i)}(x) = R_{\alpha_i, I_i}(x) +\frac{\sigma_g(i)-i}{n}$ for $x\in I_i$.
\smallskip

We define the \textbf{rotation vector} of $g$ by $\boldsymbol{\alpha_g}=(\alpha_1, ..., \alpha_n)=(\alpha_1(g), ..., \alpha_n(g)) \in {\mathbb S_n} ^n$ and we denote $\boldsymbol{g=(\alpha_g, \sigma_g)}$.
\end{defi}

\medskip

\begin{exem}
The IETs $a$ and $b$, represented above by their graphs, are expressed as: \hfill\break
(1) $a,b \in G_2$, \ \ $\left\{\begin{array}{ll} \alpha_a= (\alpha, -\alpha) \cr  \sigma_a =\id \end{array}\right\}$ \ and \ 
 $\left\{\begin{array}{ll} \alpha_b= (\beta_1, \beta_2) \cr  \sigma_b \mbox{ is the transposition } (1,2) \end{array}\right\}$. \hfill\break
(2) $a,b \in G_4$, \ \ $\left\{\begin{array}{ll} \alpha_a= (0,\alpha, -\alpha, 0) \cr  \sigma_a =(2,4)\circ(1,3)  \end{array}\right\}$ \ and \ 
 $\left\{\begin{array}{ll} \alpha_b= (\beta,0,-\beta,0) \cr  \sigma_b=(2,3)\circ(1,4)  \end{array}\right\}$.
\end{exem}

\begin{rema}\label{cons} 

A straightforward consequence of this definition is that $G_n$ is a group and composing two elements $f$ and $g$ of $G_n$, we get 
$\left\{\begin{array}{ll} \alpha_{f\circ g} = \sigma_{g} (\alpha_{f}) +\alpha_{g} \cr \sigma_{f\circ g}= \sigma_{f}\circ \sigma_{g}\end{array}\right\}$, where $\sigma_g$ acts on the vector $\alpha_f$ by permuting its coordinates.

Therefore, the map $(\alpha,\sigma ): G_n \rightarrow ({\mathbb S}_n)^n \rtimes\mathcal S_n$ is an isomorphism. In particular, the group $G_n$ is virtually abelian and {by Proposition 1.2 of \cite{DFG}}, any finitely generated virtually abelian group is isomorphic to a subgroup of some $G_n$.
\end{rema}

In Proposition \ref{freeDFG1.2}, we will specify the  construction of \cite{DFG} in order to provide free actions and conclude that the finitely generated groups that act freely as elements of $\iet$ are exactly the virtually abelian ones and they are isomorphic to subgroups of the $G_n$. 

Besides this last fact, we note that certain algebraic or dynamical properties of the subgroups of $G_n$ are particularly easy to understand and then these groups provide toy models for $\mathcal{G}$. In a ``forthcoming work", we study reversible elements of $\mathcal{G}$  beginning by a description and classification of reversible elements of $G_n$ (see \cite{GLrev} for a preliminary version).

Noticing that H\"older's Theorem can be precised: ``free actions on the circle are topologically semiconjugate to {actions of rotations groups}" (see e.g. comments following Theorem 2.2.32 in \cite{NaB}), it makes sense to look at a dynamical classification of free actions by elements of $\iet$. We will do this classification up to conjugacy by maps that are not required to be in $\iet$, actually our conjugating maps are given by

\begin{defi} \label{defPLIET} An element $f\in \iet$ is $\boldsymbol{\PL\circ \iet}$\textbf{-conjugate to } $\boldsymbol{F\in G_n}$, if $f$ is conjugate to $F$ through a map $\pl \circ E$, where  $E\in \iet$  and $\pl$ is a PL-homeomorphism such that $P^{-1}$ is affine on $[\frac{k-1}{n},\frac{k}{n})$ for any $k \in \{1,...,n\}$.
\end{defi}

Here, we will prove

\begin{theo} \label{theo1} The image in $\mathcal G$ of any free action of a finitely generated group containing a copy of $\mathbb Z^2$ is $\PL \circ \iet$-conjugate to a subgroup of some $G_n$.
\end{theo}

Our proof strategy does not strictly require actions to be free, but rather the presence of specific elements that occur when considering sufficiently large free actions. More precisely, the first step is

\begin{prop}\label{z2}
Any free action of $\mathbb Z^2$ in $\iet$ consists of periodic point free maps having an iterate that is conjugate in $\iet$ to a product of restricted rotations with disjoint supports. 
\end{prop} 
Note that this proposition implies that the image of any free action of $\mathbb Z^2$ in $\iet$  always contains maps which are conjugate in $\iet$ to products of restricted rotations with disjoint supports and without periodic points. And the second step is 
\begin{theo}\label{theo2}
Let $\G$ be a virtually abelian subgroup of $\iet$. If $\G$ contains an element that is conjugate in $\iet$ to a product of restricted rotations with disjoint supports and without periodic points, then $\G$ is $\PL \circ \iet$-conjugate to a subgroup of some $G_n$.
\end{theo}

Some extra dynamical assumptions provide stronger rigidity results:

\begin{defi} \ 

$\pbul$ An IET $f$ is {\bf totally minimal} if for all $p\in \Z^*$, the map $f^p$ is minimal. 

More generally: 

$\pbul$ An action by elements of $\iet$ of a group $\G$ is \textbf{totally minimal} if any finite index subgroup of $\G$ acts minimally on $[0,1)$.
\end{defi}

\begin{rem} According to Keane \cite{Ke}, an IET whose maximal continuity intervals have rationally independent lengths is totally minimal. 
\end{rem}

\begin{coro}\label{coro1}
Let $\G$ be virtually abelian group containing a copy of $\mathbb Z^2$. Then any faithful totally minimal action of $\G$ in $\iet$ is conjugate in $\iet$ to an action by rotations.
\end{coro}

By Schreier Lemma and the classification of finitely generated abelian groups, any finitely generated virtually abelian group is either virtually cyclic or it contains a copy of $\mathbb Z^2$. Combining this with the result of \cite{DFG} involved in Remark \ref{R1Gro}, we deduce

\begin{coro} \label{coro2}
Let $\G$ be a finitely generated subgroup of $\iet$ that contains a totally minimal IET non conjugate to a rotation then $\G$ is either virtually cyclic or not virtually nilpotent.
\end{coro} 

Using Theorem \ref{theo2} and its ingredients of proof, we can exhibit explicit examples of non virtually nilpotent subgroups of $\iet$:
\begin{coro} \label{coro3}
The group $\G$ generated by the following $f$ and $g$ is not virtually nilpotent.
\begin{enumerate}
\item  $f$ is an irrational rotation and  $g$ is not a rotation.
\item  $f\in G_n$ is periodic point free and  $g\notin G_n$. 
\end{enumerate}
\end{coro} 

For $n=1$, this is also a corollary of Proposition 5.9 of \cite{JMBMS}. In addition, in the last section, we provide examples of finitely generated  non virtually nilpotent subgroups of $\iet$ that have extra properties as being metabelian or non virtually solvable. In particular, these subgroups are generated by rotations and torsion elements 
as the groups considered by Boshernitzan in \cite{Bos}.

\begin{rema} By \cite{DFG}, all the finitely generated non virtually nilpotent subgroups of $\iet$ involved in Corollaries \ref{coro2}, \ref{coro3}, \ref{coro4} and the last section are not virtually polycyclic.
\end{rema}
\smallskip

Another motivation for studying the dynamics of the virtually abelian subgroups of $\iet$ comes from the Baumslag-Solitar group $BS(1,-1)= \langle\ a,b \ \vert \  ba b^{-1}=a^{-1}\ \rangle$ and its connections with reversible maps: O'Farrell and Short (\cite{FS}) have pointed out the importance of reversibility in dynamical systems and group theory and they have raised the following question: given a group $G$, are all reversible elements of $G$ reversible by an involution? In \cite{GLrev}, we study reversible IETs and free actions of $BS(1,-1)$ play a key role for O'Farrell and Short question. In particular, Theorem \ref{theo1} applies to such actions (since the subgroup $\langle\ a,b^2 \ \rangle$ is isomorphic to $\Z^2$ and of index $2$) leading to
\begin{coro} Any free action of $BS(1,-1)$ by elements of $\iet$ is $PL\circ \iet$-conjugate to a free action of $BS(1,-1)$ by elements of some $G_n$.
\end{coro}
 
\smallskip

\noindent \textbf{Acknowledgments.} We acknowledge support from the MathAmSud Project GDG 18-MATH-08, the Labex CEMPI (ANR-11-LABX-0007-01), ANR Gromeov, the University of Lille (BQR) and the I.F.U.M.I.


\section{Preliminaries.}

\subsection{Algebraic classification of finitely generated groups acting freely by elements of $\iet$}

\begin{prop} \label {freeDFG1.2} 
Every finitely generated virtually abelian group acts freely as a subgroup of some $G_n$. 
\end{prop}

\begin{proof}

Let $\G$ be a finitely generated virtually abelian group. By the classification of finitely generated abelian groups, there exists $H \lhd \G$ isomorphic to some $\mathbb Z^m$ such that $F=\G/H$ is finite. In addition, $H$ can be represented as a group of circle rotations $A$ seen as $2$-IETs of $[0,1)$. Namely, $A=\langle R_{\alpha_1}, ..., R_{\alpha_m}\rangle$ and $1, \alpha_1, \cdots, \alpha_m$ are $\mathbb Q$ linearly independent.

\medskip

According to Krasner-Kalaynie (see [KK]) $\G$ embeds in $H^F \rtimes F$, where $F$ acts on $H^F$ by permuting coordinates. Thus, it suffices to show that $A^F \rtimes F$ acts freely by elements of some $G_n$.

\medskip

To any $(a_s)_{s\in F} \in A^F$ and $\sigma\in F$, we associate a map  
$$E((a_s),\sigma) : \left\{ \begin{array}{ll}
[0,1) \times F \to [0,1)\times F \cr  
(x,s) \mapsto (a_s (x), \sigma s )  \end{array} \right.$$ 

Noticing that $E((a_s),\sigma) \circ  E((a'_s),\sigma') = E((a_{\sigma' s}\circ a_s'),\sigma \sigma')$, we get that $E$ induces an injective morphism from $A^F \rtimes F$ to the group $\PIET([0,1)\times F):=\{ E((a_s),\sigma), \ (a_s)\in (\iet_2) ^F, \ \sigma\in F \} $.

\medskip

In addition, its image acts freely on $[0,1)\times F$ because $1, \alpha_1, ..., \alpha_m$ are $\mathbb Q$ linearly independent.

Finally, $\PIET([0,1)\times F)$ embeds in $G_{\#F}$ by subdividing $[0,1)$ in $\# F$ intervals of same length. \end{proof}

\subsection{Dynamical properties of IETs.} \label{DynIET} \ 

\smallskip

\begin{defi} \ 

$\pbul$ The \textbf{break point set} of an IET $f$ is obtained by adding the initial point $0$ to the discontinuity set of $f$, it is denoted by $\boldsymbol{\BP(f)}$. The set consisting of the $f$-orbits of points in $\BP(f)$ is denoted by $\boldsymbol{\BP_{\infty}(f)}$. The \textbf{translation set} of $g$ is $\{g(x)-x, \ x\in [0,1)\}$.  

$\pbul$ The \textbf{break point set} of a finitely generated subgroup $G$ of $\iet$ is $\ds\boldsymbol{\BP(G)}:=\bigcup_{g\in G} \BP(g)$ 

\vskip -3mm \noindent and its \textbf{translation set} is $\{g(x)-x, \ g\in G, \ x\in [0,1)\}$.  
\end{defi}

\smallskip

We let the reader check the following
\begin{propri}\label{proBP} Let $f$, $g$ in $\iet$ and $n\in \N$.
\begin{enumerate}[(a)]  
\item $\BP(f\circ g) \subseteq \BP(g)\cup g^{-1}(\BP(f))$,
\item $\BP (f^{-1})=f(\BP(f))$ and
\item $\BP (f^n) \subseteq \BP (f)\cup f^{-1}(\BP(f))\cup ...\cup  f^{-n+1}(\BP(f)).$
\end{enumerate}
\end{propri}

\begin{rema} \label{remBouBP} By these properties, $\#\BP(f^n)$ is bounded above by $\#\BP(f) \times \vert n \vert$. However, it is possible for $\#\BP(f^n)$ to be bounded independently of $n$, for IETs of finite order but also for minimal IETs as irrational rotations and then also for all elements of $G_n$, in this case we say that $f$ is of \textbf{bounded break point type}. 
\end{rema}

\begin{defi} Let $f\in \iet$.

$\pbul$ Let $p\in \N^*$, we denote by ${\Per_p}(f)$ the set of points $x$ such that, $\boldsymbol{\mathcal O_f(x)}$, the $f$-orbit of $x$ has cardinality $p$. The fix point set of $f$ is $\fix (f)={\Per_1}(f)$.

$\pbul$ Let $V$ be a subset of $[0,1)$, we say that $V$ is \textbf{of type $\mathcal M$} if it is a non empty finite union of intervals each of the form $[b,c)$, with $b,c$ in $\BP_{\infty}(f)$. A type $\mathcal M$ and $f$-invariant set that is minimal for the inclusion among non empty, type $\mathcal M$ and $f$-invariant subsets of $[0,1)$ is called an \textbf{$f$-component}.
\end{defi}

\smallskip

The well-known decomposition into minimal and periodic components of an IET was first given for measured surface flows by Mayer in 1943 (\cite{May}) and restated for IETs by Arnoux (\cite{Ar}) and Keane (\cite{Ke}).

\medskip

\textbf{The Arnoux-Keane-Mayer decomposition Theorem} claims that $[0,1)$ can be decomposed as $[0,1)= P_1\cup ...P_l \cup M_1...\cup M_m,$ where

\begin{itemize}
\item $P_i$ is an $f$-periodic component: $P_i$ is the $f$-orbit of an interval $[b,c)$ with $b,c$ in $\BP_{\infty} (f)$ and all iterates $f^k$ of $f$ are continuous on $[b,c)$. In particular points in $P_i$ are periodic of the same period. 
\item $M_j$ is an $f$-minimal component: for any $x\in M_j$, the orbit $\mathcal O_f(x)$ is dense in $M_j$.
\end{itemize}

For convenience, the $f$-component containing $x\in I$ will be denoted $\boldsymbol{M_f(x)}$. 

\medskip

According to Remark \ref{remBouBP}, it makes sense to give the following 

\begin{defi} We define the growth rate of the number of discontinuities for the iterates of an IET $f$ on the $f$-orbit through a given point $x$ by
$$\boldsymbol {{\mathcal N}_x(f)}= \lim \limits_{n \rightarrow + \infty} \frac{\# \{ \BP(f^n) \cap {\mathcal O}_x(f)\}}{n}.$$
\end{defi}

\begin{propri}\label{inva} Let $f$, $g$ in $\iet$ and $x\in [0,1)$. 
\begin{enumerate}
\item If $BP(f) \cap {\mathcal O}_x(f)=\emptyset$ then ${\mathcal N}_x(f)=0$.
\item ${\mathcal N}_{g(x)}(g \circ f \circ g^{-1})={\mathcal N}_x(f).$
\end{enumerate}
\end{propri}

\begin{proof} \ 

Item (1) follows from Item (c) of Properties \ref{proBP}.

\medskip

Item (2). Let $N\in \N$, we write $\boldsymbol{A\nearsubset\limits_N B}$ if $A\subset B \cup F$  where $\#F\leq N$.

Let $C= 2 \# \BP(g)$ and $f_g=g\circ f \circ g^{-1}$, we have ${\mathcal O}_{g(x)}(f_g)=g({\mathcal O}_{x} (f))$.

\medskip

It follows from Item (a) of Properties \ref{proBP} that:

\smallskip

\noindent $\pbul$ $\ds \BP(f_g ^n)\subseteq \BP (g^{-1})\cup g(\BP (f^n))\cup g\circ f^{-n}(\BP (g))\nearsubset\limits_C g(\BP (f^n))$
and

\noindent $\pbul$ $\ds \BP(f^n)=\BP(g^{-1} \circ f_g ^n \circ g) \subseteq \BP(g)\cup g^{-1}(\BP(f_g^n)) \cup g^{-1} \circ f_g^{-n}(\BP(g^{-1}))\nearsubset\limits_C  g^{-1}(\BP(f_g^n))$.

Therefore $\ds  \BP(f_g ^n)\cap {\mathcal O}_{g(x)}(f_g) \nearsubset\limits_C g\bigl(\BP(f^n) \cap g^{-1}({\mathcal O}_{g(x)}(f_g)) \bigr)= g\bigl(\BP (f^n) \cap {\mathcal O}_{x} (f)\bigr)$ 

and $\ds \BP(f^n) \cap {\mathcal O}_{x} (f)\nearsubset\limits_C   g^{-1}(\BP(f_g^n) \cap g({\mathcal O}_{x} (f)) =  g^{-1}(\BP(f_g^n) \cap {\mathcal O}_{g(x)}(f_g)).$ 

Finally, we get $$ \# \left(\BP(f^n) \cap {\mathcal O}_{x}(f) \right)- C \leq  
\# \left( \BP(f_g^n) \cap {\mathcal O}_{g(x)}(f_g)\right) \leq \# \left( \BP(f^n) \cap {\mathcal O}_{x} (f) \right)+ C.$$

We conclude by dividing by $n$ and taking the limit.
\end{proof}

In addition, ${\mathcal N}_x(f)$ provides a characterization of the bounded break point type. More precisely, from \cite{No}, we deduce

\begin{theor}\textrm{Adapted version of Theorems 1.1 and 1.2 of \cite{No}.} \label{Novak+} 

Let $f\in \iet$, the following assertions are equivalent
\begin{enumerate}
\item For any $x\in I$, we have ${\mathcal N}_x(f)=0$.
\item The map $f$ is of bounded break point type.
\item There exists a positive integer $p$ such that $f^p$ is conjugate in $\mathcal G$ to a product of restricted rotations of pairwise disjoint supports.
\item There exist two positive integers $p$ and $n$, an IET $E$ and a PL-homeomorphism $\pl: I \rightarrow I$ such that $\Phi=( \pl \circ E) \circ f^p \circ (\pl \circ E)^{-1} \in G_n$, $\sigma_\Phi=\id$ and the map $\pl ^{-1}$ is affine on each $I_i=[ \frac{i-1}{n}, \frac{i}{n})$. 
\end{enumerate}
\end{theor}

\begin{proof} \ 

Since $\displaystyle \BP(f^n)=\bigcup_{a\in \BP(f)} \BP(f^n) \cap {\mathcal O}_{f}(a)$, the implication $(1) \Ra (2)$ is a direct consequence of Theorem 1.1 of \cite{No} stating that either $\#\BP(f^n)$ is bounded or $\#BP(f^n)$ has linear growth.

\smallskip

The implication $(2) \Ra (3)$ is exactly Theorem 1.2 of \cite{No}.

\smallskip

We prove $(3) \Ra (4)$. Let $E$ be an IET that conjugates $f^p$ to $F$, a product of restricted rotations of pairwise disjoint supports. We decompose $I$ as a disjoint union of consecutive half open intervals $J_i$, $i=1,...,n$, where $J_i$ is either the support of a restricted rotation in $F$ or a connected component of $\fix(F)$. We define $\pl$ as the PL-homeomorphism such that $\pl(J_i)=I_i$ and $\pl_{\vert J_i}$ is affine. It is easily seen that $\Phi=\pl\circ F \circ \pl^{-1} \in G_n$ and $\sigma_\Phi=\id$.

\smallskip

It remains to prove $(4) \Ra (1)$. As $\Phi \in G_n$, it holds that ${\mathcal N}_x(\Phi)=0$ for any $x\in I$, then ${\mathcal N}_x(\pl^{-1}\circ\Phi\circ\pl)=0$ since $\pl$ is a homeomorphism. Therefore Item (2) of Properties \ref{inva} implies that ${\mathcal N}_x(f^p)=0$ for any $x\in I$. It follows that ${\mathcal N}_x(f)=0$ for any $x\in I$. Indeed : Let $x\in I$ and $n\in \N^*$, we have $\displaystyle {\mathcal O}_{f}(x) =\bigcup_{j=0}^{p-1} \mathcal O_{f^p}(f^j(x))$ and then 
$$\frac{\# \bigl( \BP(f^{pn}) \cap {\mathcal O}_f(x) \bigr)}{pn} \leq \frac{1}{p}\sum_{j=0}^{p-1} 
\frac{\# \bigl( \BP(f^{pn}) \cap {\mathcal O}_{f^p}(f^j(x))\bigr)}{n}$$
Taking, the limit when $n$ goes to infinity, we get that 
$\displaystyle  {\mathcal N}_x(f)  \leq \frac{1}{p} \sum_{j=0}^{p-1} {\mathcal N}_{f^j(x)}(f^p).$

Finally, by Item (2) of Properties \ref{inva}, it holds that ${\mathcal N}_x(f^p)=  {\mathcal N}_{f^j(x)}(f^p)$ for any integer $j$ and then ${\mathcal N}_x(f) \leq \ {\mathcal N}_x(f^p)=0$.
\end{proof}

\subsection{Centralizers.}

\subsubsection{Generalities} \

Let $\G$ be a group and $f\in \G$. It is well-known that the \textbf{centralizer} of $f$ in $\G$, $\ds \boldsymbol{\CRd(f)}:=\{ h\in \G \ : \ h\circ f  = f\circ  h \}$, is a subgroup of $\G$.

\begin{lemm}\label{DynGen} {Let $f,h$ be bijections of a set $X$ and $h\in \CR(f)$}, then:  
\begin{itemize}
\item[$\pbul$] For any $x\in X$, $h(\mathcal O_f(x)) = \mathcal O_f( h(x))$.
\item[$\pbul$] For any $p\in \mathbb N^*$, the set of $f$-periodic points with exact period $p$, ${\Per}_p(f)$, is $h$-invariant.
\end{itemize}
\end{lemm}

\begin{proof} \ 

$\pbul$ The first item is a direct consequence of the fact that for any $x\in X$, one has $h(f^{n}(x))=f^{n}(h(x))$.

\smallskip

$\pbul$ Let $p\in \mathbb N^*$, $x\in {\Per}_p(f)$ means that $\# \mathcal O_f(x) = p$. By the first item, we get that $\# \mathcal O_f(h(x)) = p$ and therefore $h(x) \in {\Per}_p(f)$.
\end{proof}

\begin{lemm} \label{DynIETCR} Let $f,h\in \iet$, $h \in \CR(f)$ and $x$ be a non $f$-periodic point. 

Then $h(x)$ is not $f$-periodic and $h(M_f(x))= M_f (h(x))$. In addition, there exists a positive integer $r$ such that $h^r(M_f(x))= M_f(x)$. 
\end{lemm}

\begin{proof} As $x$ is not $f$-periodic, the orbit $\mathcal O_f(x)$ is dense in $M_f(x)$. By Lemma \ref{DynGen}, the point $h(x)$ is not $f$-periodic and $h(\mathcal O_f(x))= \mathcal O_f( h(x))$, taking closure we get $h(M_f(x)) = M_f (h(x))$.\footnote{\ The map $h$ may not be continuous. However the dynamics of the group $G$ generated by $f$ and $h$ can be lifted, by a Denjoy blow-up, to that of a homeomorphisms group of $X =[0,1) \sqcup \BP^-_\infty(G)$, where $\BP^-_\infty(G)$ is an abstract copy of $\BP_\infty(G)$. To any $x \in \BP_\infty(G)$ correspond two points of $X$: $x^+ \in [0,1)$ and $x^- \in   \BP^-_\infty(G)$. The topology on $X$ is given by the total order defined on $X$ by taking the usual order on  $\BP^-_\infty(G)$ and $[0,1)$ and setting that for  $x^- \in \BP^-_\infty(G)$ and  $y \in  [0,1)$:  $x^-< y \LRA x^+ \leq  y$. For more details see e.g \cite{Ar} or \cite{Cor} $\S 3.2$.}  In addition, the finiteness of the number of $f$-minimal components implies that there exist two distinct integers $s$ and $t$ such that $h^s(M_f(x))=h^t(M_f(x))$ therefore $h^{t-s}(M_f(x))=M_f(x)$.
\end{proof}

\subsubsection{Centralizer for IETs of bounded break point type.} \ 

We begin by recalling

\begin{propri}\label{conjrot} (\cite{No}, Lemma 5.1 and its proof.)
\begin{enumerate}
\item Two irrational rotations $R_\alpha$ and $R_\beta$, with $\alpha\not=\beta \mod 1$, are not conjugate in $\iet$. 
\item The centralizer in $\iet$ of an irrational rotation is the rotation group $\ietd$.
\end{enumerate}
\end{propri}

This extends to 
\begin{lemm}\label{VirtGn}
Let $F\in G_n$ without periodic points then $\CR(F)\subset G_n$.
\end{lemm}
 
\begin{proof} Let $F\in G_n$ without periodic points, noting that $C(F)\subset C(F^n)$ holds for all $n\in \Z$ and after replacing $F$ by some iterate $F^p$, it may be assumed that $\sigma_{F}=\id$ and then the intervals $I_i=[\frac{i-1}{n},\frac{i}{n})$ are the minimal components of $F$, since $F$ has no periodic points.

\smallskip

Let $H\in \CR(F)$. By Lemma \ref{DynIETCR}, there exists a permutation $\gamma \in \mathcal S_n$ such that $H(I_i)= I_{\gamma(i)}$ and $H$ conjugates the irrational rotation $F_{\vert I_i}$ to the irrational rotation $F_{\vert I_{\gamma(i)}}$. Thus, Item (1) of Properties \ref{conjrot} implies that $\alpha_{\gamma(i)} (F) =\alpha_{i}(F)$.

\smallskip

Next, we consider the map $K$ of $G_n$ defined by $\alpha (K)=0$ and $\sigma_K=\gamma$. It is easily seen that $K \in \CR(F)$, therefore $K^{-1}\circ H$ is an IET that commutes with $F$ and fixes each $I_i$. By Item (2 )of Properties \ref{conjrot}, one has $(K^{-1}\circ H)_{\vert I_i}= R_{\beta_i}$ for any $i$ meaning that $K^{-1}\circ H\in G_n$, so does $H$.
\end{proof}

More generally, we have 

\begin{prop}\label{CGn} Let $f$ be a periodic-point free IET that is conjugate to $F\in G_n$ by a PL-homeomorphism $\pl$ such that $\pl^{-1}_{\vert  I_i}$ is an orientation preserving affine map. Then the map $\pl$ conjugates $\CR(f)$ to a subgroup of $G_n$.
\end{prop}

\begin{proof}

As in the proof of Lemma \ref{VirtGn}, the map $f$ has no periodic points then after replacing $f$ by some iterate, we can suppose that the intervals $J_i := \pl^{-1} (I_i)$ are the minimal components of $f$. Let $g \in \CR(f)$, according to Lemma \ref{DynIETCR}, there exists a permutation $\gamma \in \mathcal S_n$ such that $g(J_i)= J_{\gamma(i)}$. In particular $J_i$ and $J_{\gamma(i)}$ have same lengths and then the restrictions of $\pl$ to $J_i$ and to $J_{\gamma (i)}$ are affine with same slopes. 

This implies that the a priori AIET $G=\pl \circ g \circ \pl ^{-1}$, which sends $\pl(J_i)=I_i$ to $\pl(J_{\gamma(i)})=I_{\gamma(i)}$ for any $i$, is actually an IET. Finally $G$ is an IET that commutes with $F$ (i.e. $G\in \CR(F)$), therefore, by Lemma \ref{VirtGn}, $G$ belongs to $G_n$. \end{proof}

\section{Free actions of $\mathbb Z^2$ by IET} The aim of this section is to prove Proposition \ref{z2}. By  Theorem \ref{Novak+}, it can be reformulated as: \textit{``Any free action of $\mathbb Z^2$ in $\iet$ consists of IETs of bounded break point type."}

\smallskip

This proof, inspired by the description given by Minakawa for the free actions of $\mathbb Z^2$ by circle PL-homeomorphisms (see \cite{Mi}), uses the properties of the rates ${\mathcal N}_x(f)$.

\medskip

As the action $\langle f,h\rangle$ of $\mathbb Z^2$ is free, for every $x\in I$ and $k,q$ distinct integers, one has: $${\mathcal O}_{f}(h^k(x)) \cap {\mathcal O}_{f}(h^q(x)) = \emptyset.$$
Therefore, as $\# BP(f)$ is finite, there exists a positive integer $N_0$ such that: $$\forall n\geq N_0, \ \  {\mathcal O}_{f}(h^n(x)) \cap BP(f) = \emptyset.$$
Using Item (1) of Properties \ref{inva}, one get ${\mathcal N}_{h^n(x)}(f)=0$ for $n\geq N_0$ and, by Item (2), $$ {\mathcal N}_{h^n(x)}(h^n \circ f \circ h^{-n})={\mathcal N}_x(f).$$
\noindent By {commutativity}, it holds that $ {\mathcal N}_{h^n(x)}(h^n \circ f \circ h^{-n})={\mathcal N}_{h^n(x)}(f)$. 

\medskip

Combining these properties, we get ${\mathcal N}_x(f)=0$ for any $x \in I$. Finally, by Theorem \ref{Novak+}, $f$ has bounded break point type.

\section{Proof of Theorem \ref{theo2}.}

Eventually conjugating by an IET, it suffices to prove 

\begin{prop}\label{ext2}
Let $G$ be a virtually abelian subgroup of $\mathcal G$. If $\G$ contains a product of restricted rotations with disjoint supports and without periodic points, $f_1$, then $\G$ is conjugate to a subgroup of some $G_n$ by the canonical PL-homeomorphism that conjugates $f_1$ to an element of $G_n$.
\end{prop}

\subsection{Proof of Proposition \ref{ext2}.} \

\smallskip
Let $\G<\iet$ and $H\lhd \G$ be an abelian subgroup of $\G$ with finite index, $p$. 

\begin{claim}\label{power} To any $f\in \G$ we can associate a positive integer $p_f$ in order that the maps $f^{p_f}$ are pairwise commuting.
\end{claim}
\begin{quote}
Indeed, let $f\in \G$, since the set $G/H$ is finite, there exist integers $0\leq s < t$ such that $f^s$ and $f^t$ belong to a same class modulo $H$, that is $f^{t-s} \in H$. So taking $p_f=t-s$ leads to the required conclusion.
\end{quote}

\smallskip

Consider $\pl$ the PL-homeomorphism associated to $f_1$ as involved for proving $(3) \Ra (4)$ of Theorem \ref{Novak+}. In particular, the $f_1$-components are the intervals $J_i=\pl^{-1}(I_i)$ for $i=1,...,n$ ; they are minimal and $\pl_{\vert J_i}$ has constant slope for any $i\in \{1,\cdots , n\}$.

\smallskip

Let $f\in G$ and $\phi = f \circ f_1 \circ f^{-1}$. By Claim \ref{power}, there exist $p_{f_1}, p_\phi \in \N^*$ such that $f_1^{p_{f_1}}$ and $\phi^{p_\phi}$ commute. Thus,  Proposition \ref{CGn} applies to $\phi^{p_\phi} \in \CR(f_1^{p_{f_1}})$ and implies that $\phi ^{p_{\phi}}$ is conjugate by $\pl$ to an element of $G_n$, namely $\pl \circ \phi ^{p_{\phi}} \circ \pl^{-1}\in G_n$. 

Moreover, as $\phi ^{p_{\phi}}$ is periodic points free, eventually replacing $p_{\phi}$ by a multiple, we may assume that the components of $\phi ^{p_{\phi}}$ are the intervals $J_i$ and they are minimal. 
\smallskip

Therefore, as $f$ conjugates $\phi ^{p_{\phi}}$ and $f_1^{p_{\phi}}$, the map $f$ permutes the $J_i$. This implies that the a priori AIET $\pl \circ f \circ \pl^{-1}$ is an IET. Indeed, for any $i=1,...,n$, it holds that
\begin{itemize}
\item[$\pbul$] the map $\pl^{-1}$ sends $I_i$ to $J_i$ with constant slope $\frac{\vert J_i \vert}{\vert I_i \vert}$, 
\item[$\pbul$] the IET $f$ sends $J_i$ to some $J_k$ of same length, with constant slope $1$ and
\item[$\pbul$] the map $\pl$ sends $J_k$ to $I_k$ with constant slope $\frac{\vert I_k \vert}{\vert J_k \vert}= \frac{\vert I_i \vert}{\vert J_i \vert}$.
\end{itemize}

\smallskip

Now, it remains to prove that $\pl \circ f \circ \pl^{-1}\in G_n$.

\begin{lemm} \label{conjGn} Let $n$ be a positive integer and $T\in G_n$ be a periodic point free IET. If $S\in \mathcal G$  is such that $T'=S T S^{-1}\in G_n$ then $S\in G_n$.
\end{lemm}

\begin{proof} \ 

$\bullet$ If $n=1$ then $T$ and $T'$ are irrational rotations and $T'S=ST$. Identifying $[0,1)$ with $\mathbb S^1$, the maps $T$ and $T'$ become continuous and computing break point sets, we get 
$$\BP(S)=\BP(T'S)=\BP(ST)=T^{-1}(\BP(S)).$$
Thus $\BP(S)$ is a finite $T$-invariant set so the minimality of $T$ implies that it is empty and then $S$ is a rotation.
 
\medskip
 
$\bullet$ If $n\geq 2$ then eventually changing $T$ for some iterate, we can suppose that $\sigma_T$ and $\sigma_{S T S^{-1}}$ are trivial.

\begin{itemize}
\item[$\pbul$] If $S$ preserves each $I_j$, we apply the case $n=1$ to the restrictions to $I_j$ of $T$ and $S$.
We conclude that $S\vert_{I_j}$ is a rotation and then $S\in G_n$. 
\item[$\pbul$] If $S$ does not preserve each $I_j$. As $S$ conjugates two periodic point free maps in $G_n$ there exists a permutation $\gamma$ such that $S(I_j)=I_{\gamma(j)}$.  We consider the element $K\in G_n$ defined by $\sigma_K=\gamma$ and $\alpha_K=0$. We have $ K^{-1} S \circ T \circ S^{-1} K \in G_n$ and $K^{-1} S$ preserves each $ I_j$. According to the previous case, $K^{-1} S\in G_n$ and therefore $S\in G_n$. \end{itemize} 

\vskip -0.6cm \end{proof}

We conclude by applying Lemma \ref{conjGn} to $T= \pl \circ f_1^{p_{\phi}} \circ \pl^{-1} \in G_n $ and $S=\pl \circ f \circ \pl^{-1}$.

\section{Proofs of Theorem \ref{theo1} and corollaries \ref{coro1} and \ref{coro2}.}

\noindent{\bf Proof of Theorem \ref{theo1}}. Let $\G$ be a finitely generated group acting freely by IET and $\Gamma<\iet$ be the image of a copy of $\mathbb Z ^2$ contained in $\G$. According to Remark \ref{R1Gro}, $\G$ is virtually abelian. It is plain that $\Gamma$ also acts freely, then Proposition \ref{z2} implies that $\Gamma$ contains IETs that are conjugate in $\iet$ to products of restricted rotations with disjoint supports and without periodic points. Therefore, $\G$ satisfies the hypothesis of Theorem \ref{theo2} and then its conclusion.

\bigskip

\noindent{\bf Proof of Corollary \ref{coro1}}. Let $\G$ be a virtually abelian subgroup of $\iet$ containing a copy $\langle a,b \rangle$ of $\mathbb Z^2$ and whose action is totally minimal. 

Let $K$ be a finite index abelian normal subgroup of $\G$. Claim \ref{power} implies that there exist integers $p$ and $q$ such that $a^p$ and $b^q$ generate a copy $Z$ of $\mathbb Z^2$ in $K$. We claim that that $Z$ acts freely. 

\begin{quote} Indeed, we argue by contradiction supposing that there exist $h\in Z \setminus\{\id\}$ and $x\in I$ such that $x\in \fix(h)$. As $K$ is abelian, Lemma \ref{DynGen} implies that $\fix(h)$ is a non empty $K$-invariant set. By totally minimality of $G$, all $K$-orbits are dense, hence $\overline{\fix(h)} = [0,1)$ and then $h=\id$, a contradiction.
\end{quote}
\smallskip

Therefore, by Proposition \ref{z2}, $\G$ contains maps which are conjugate in $\iet$ to products of restricted rotations with disjoint supports and without periodic points. Let $f_1$ be such an element of $\G$, Theorem \ref{theo2} applies to $\G$ and we conclude that there exists a map $H=P\circ E$ that conjugates $\G$ to a subgroup of some $G_n$, with $P$ a PL-homeomorphisms, $E\in\iet$  and  $H$ of constant slope on each minimal component of $f_1$. 

Moreover, as the conjugation by $H$ preserves the totally minimality (according to classical arguments for topological conjugacy and extra arguments given in the proof of Lemma \ref{DynIETCR} and its footnote), this subgroup is also totally minimal.

\medskip

We claim that no subgroup of $G_n$, $n\geq 2$, is totally minimal.
\begin{quote} Indeed, if $\G<G_n$ then $G_0=\{g \in \G \ | \ \sigma_g=\id \}$ is a finite index subgroup of $G$ that preserves any $I_i$, $i=1,...,n$.  Thus, for $G_0$ to be minimal it is necessary that $n=1$.
\end{quote} 
Finally $n=1$, so the map $f_1$ is minimal and then $H$ has constant slope $1$. In conclusion, $\G$ is conjugate in $\iet$ to a subgroup of $G_1$.

\medskip

\noindent{\bf Proof of Corollary \ref{coro2}}. Let $\G$ be a finitely generated subgroup of $\iet$ and $f\in \G$ be a totally minimal IET that is not conjugate to a rotation. 

We first note that the action of $\G$ is also totally minimal, since any finite index subgroup of $\G$ contains some non trivial power of $f$.

We next argue by contradiction, supposing that $\G$ is virtually nilpotent and not virtually cyclic. By \cite{DFG}, $\G$ is virtually abelian and not virtually cyclic then it contains a copy of $\mathbb Z^2$. Therefore, applying Corollary \ref{coro1}, we get that $f$ is conjugate in $\iet$ to a rotation, a contradiction.

\section{Non virtually abelian subgroups of \rm{$\iet$}.} 

\subsection{Proof of Corollary \ref{coro3}.} \ 

Item (1) is the special case $n=1$ in Item (2), so  we only prove Item (2): We argue by contradiction supposing that $G$ is virtually abelian. Hence, according to Claim \ref{power}, there exist integers $p$ and $q$ such that $f^p\in G_n$ and $g f^q g^{-1}$ commute. Therefore $g f^q g^{-1} \in G_n$ by Lemma \ref{VirtGn} and then Lemma \ref{conjGn} leads to $g\in G_n$, a contradiction.

\subsection{Related examples involving arithmetic conditions.} \ 

\begin{nota}
Let $s\in \mathbb N^*$ and $\alpha_1, ..., \alpha_s$ be $s$ $\mathbb Q$-independent irrational numbers. The additive subgroup of $\mathbb R$ generated by $1$ and the $\alpha_i$'s is denoted by $\langle 1,\alpha_1,...,\alpha_s \rangle_{\mathbb Q}$. We set
\begin{itemize}
\item $\boldsymbol{\mathcal G_{\mathbb Q}} = \left\{g\in \iet \ | \ g(x)-x \in \mathbb Q, \ \forall x\in I  \right\}$  and 
\item $\boldsymbol{\mathcal G_{\alpha_1,...,\alpha_s}} = \left\{g\in \iet \ | \ g(x)-x \in \langle 1,\alpha_1,...,\alpha_s \rangle_{\mathbb Q}, \ \forall x\in I \right\}$.
\end{itemize}
\end{nota}

\begin{coro} \label{coro4} Let  $p,r \in \mathbb N^*$, $r<p$ and $\alpha_1, ..., \alpha_p$ be $p$ $\mathbb Q$-independent irrational numbers. If $f\in \mathcal G_{\alpha_1,...,\alpha_r}$ is totally minimal not conjugate to a rotation and  $g\in \mathcal G_{\alpha_{r+1},...,\alpha_p}$ then the group generated by $f$ and $g$ is not virtually nilpotent.
\end{coro} 

\begin{proof} According to Corollary \ref{coro2}, $\G$ is either virtually cyclic or not virtually nilpotent. Suppose, contrary to our claim that $\G$ is virtually cyclic. Hence, there exists $h\in \G$ such that $Z=\langle h \rangle$ has finite index in $\G$ and then $f^m = h ^n$ and $g^q = h^t$ for some  positive integers $m$, $n$, $q$ and $t$. Therefore $f^{mt} = h ^{nt}= g^{qn} \in \mathcal G_{\alpha_1,,...,\alpha_r}\cap \mathcal G_{\alpha_{r+1},...,\alpha_p}=\mathcal G_{\mathbb Q}$ that only contains finite order maps \footnote{Let $f\in \mathcal G_{\mathbb Q}$  and $q$ be the least commun multiple of the deniminators of the translations of $f$, then $\mathcal O_f(x) \subset \{ x+ \frac{p}{q}, \ p\in \Z \} \cap [0,1)$ is finite.}, in particular $f$ is not totally minimal, a contradiction. Finally, $\G$ is not virtually nilpotent.
\end{proof} 
 
In the following subsections, we exhibit examples of finitely generated non virtually nilpotent subgroups of $\iet$ that have extra properties. In the first one, we provide metabelian examples and in the second one non virtually solvable groups.

\subsection{Metabelian examples.} \label{ExMeta} \

Let $\G$ be the subgroup of $\iet$ generated by $R_{\alpha}$, $\alpha \in [0, \frac{1}{3})\setminus \Q$ and the map $g\in G_3$ described by the following picture.

\vskip -0.7cm
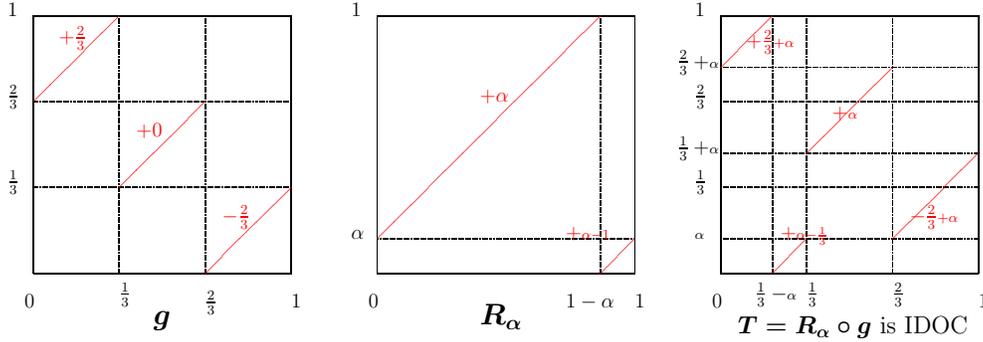
\begin{figure}[H] 
\scalebox{.65}{
\begin{picture}(560,170)
\put(0,0){\line(1,0){150}} \put(0,150){\line(1,0){150}}
\put(0,0){\line(0,1){150}} \put(150,0){\line(0,1){150}}

\put(-5,-20){$0$}
\put( 50,-15){$\frac{1}{3}$}
\put(100,-20){${\frac{2}{3}}$}
\put(150,-20){$1$}

\put(-15,50){$\frac{1}{3}$}
\put(-15,100){$\frac{2}{3}$}
\put(-15,150){$1$}

\put(50,0){\dashbox(0,150){}}
\put(100,0){\dashbox(0,150){}}

\put(0,50){\dashbox(150,0){}}
\put(0,100){\dashbox(150,0){}}

\put(0,100){\textcolor{red}{\line(1,1){50}}}
\put(15,134){\textcolor{red}{$+\frac{2}{3}$}}
\put (50,50){\textcolor{red}{\line(1,1){50}}}
\put(60,80){\textcolor{red}{$+0$}}
\put (100,0){\textcolor{red}{\line(1,1){50}}}
\put(110,30){\textcolor{red}{$-\frac{2}{3}$}}

\put (70,-30){\Large{$\boldsymbol{g}$}}

\put(200,0){\line(1,0){150}} \put(200,150){\line(1,0){150}}
\put(200,0){\line(0,1){150}} \put(350,0){\line(0,1){150}}

\put(195,-20){$0$}
\put(310,-20){$1-\alpha$}
\put(350,-20){$1$}

\put(185,20){$\alpha$}
\put(185,150){$1$}

\put(330,0){\dashbox(0,150){}}

\put(200,20){\dashbox(150,0){}}

\put(200,20){\textcolor{red}{\line(1,1){130}}}
\put(260,100){\textcolor{red}{$+\alpha$}}
\put (330,0){\textcolor{red}{\line(1,1){20}}}
\put(310,20){\textcolor{red}{$+\pal-1$}}

\put (260,-30){\Large{$\boldsymbol{R_{\alpha}}$}}

\put(400,0){\line(1,0){150}} \put(400,150){\line(1,0){150}}
\put(400,0){\line(0,1){150}} \put(550,0){\line(0,1){150}}

\put(395,-20){$0$}
\put( 420,-15){$\frac{1}{3}-\pal$}
\put( 450,-15){$\frac{1}{3}$}
\put(500,-15){${\frac{2}{3}}$}
\put(550,-20){$1$}

\put(385,20){$\pal$}
\put(385,50){$\frac{1}{3}$}
\put(375,70){$\frac{1}{3}+\pal$}
\put(385,100){$\frac{2}{3}$}
\put(375,120){${\frac{2}{3}+\pal}$}
\put(385,150){$1$}

\put(430,0){\dashbox(0,150){}}
\put(450,0){\dashbox(0,150){}}
\put(500,0){\dashbox(0,150){}}

\put(400,20){\dashbox(150,0){}}
\put(400,50){\dashbox(150,0){}}
\put(400,70){\dashbox(150,0){}}
\put(400,100){\dashbox(150,0){}}
\put(400,120){\dashbox(150,0){}}

\put(400,120){\textcolor{red}{\line(1,1){30}}}
\put(415,132){\textcolor{red}{$+\frac{2}{3}\pp \pal$}}
\put(430,0){\textcolor{red}{\line(1,1){20}}}
\put(435,20){\textcolor{red}{$+\pal \pmi \frac{1}{3}$}}
\put (450,70){\textcolor{red}{\line(1,1){50}}}
\put(465,90){\textcolor{red}{$+ \pal$}}
\put (500,20){\textcolor{red}{\line(1,1){50}}}
\put(510,30){\textcolor{red}{$-\frac{2}{3} \pp \pal$}}

\put (410,-35){\large{$\boldsymbol{T=R_{\alpha} \circ g}$ is IDOC}}

\end{picture}
} 
\end{figure} \ 

Note that, since $g$ has order $2$, any $f\in G$ can be written as 
$$f= R_\alpha^{p_s} \ g  \  R_\alpha^{p_{s-1}} \  g  \  \cdots \  R_\alpha^{p_1} \  g \  R_\alpha^{p_0}\ , \mbox{ \quad where } p_j \in \Z.$$

\begin{prop} \label{PNex}
The group $G$ is metabelian and non virtually nilpotent. More precisely its commutators subgroup $[G,G]$ is a non finitely generated abelian group whose non trivial elements have order $3$. 
\end{prop}
\begin{lemm} \label{L1Nex} For any $f \in G$, 
there exists a unique $\ell (f)  \in \mathbb Z$ such that for all $x\in [0,1)$, 
$$f(x) = x + \  \ell (f) \ \alpha \ + \frac{p_f(x)}{3} \mbox{\ , \quad  where } p_f(x) \in \mathbb Z.$$
\end{lemm}

\begin{proof} We argue by induction on the length $L_S(f)$ of $f\in G$ as a word in $S=\{g,R_\alpha\}$.

Indeed, if $L_S(f)=1$ then either $f=g$ or $f= R_{\pm \alpha}$ and the property holds. 

\smallskip

Let $n\in \N^*$, suppose that $G$-elements of length less or equal than $n$ have the required property. Consider $f\in G$ with $L_S(f)=n+1$, then either $f= R_{\pm \alpha} f_0$ or $f= g f_0$, where $L_S(f_0)=n$. So, for any $x\in [0,1)$, we have either
$$f(x) = f_0 (x) \pm \alpha + \epsilon  \mbox{\quad  \ or \  \quad}  f(x) = f_0 (x) + \epsilon \ \frac{2}{3} \ , \qquad \mbox{ where } \epsilon\in \{-1,0,1\}.$$
Therefore, either  $$f(x) = x + \  (\ell (f_0)\pm 1)\ \alpha \ + \frac{p_{f_0}(x)+ {3\epsilon}}{3}  \mbox{\quad  or \quad } 
f(x) =  x +  \ \ell (f_0) \ \alpha \ + \frac{p_{f_0}(x)+ 2 \epsilon}{3}.$$
In both cases, $f$ has the required form and $\ell(f)$ is unique since $\alpha$ is irrational.
\end{proof}

\begin{lemm} \label{L2Nex} \ 
\begin{enumerate}
\item The map $\ell : G \to \mathbb Z$ is a morphism.
\item Any $f\in G$ can be written as $f= P_f \ R_{\alpha} ^{\ell(f)}=R_{\alpha}^{\ell(f)} \ Q_f$ with $P_f,\ Q_f \in \ker \ell$.
\item The $\ker \ell$-orbit of any $x\in [0,1)$ consists in at most $3$ points. In particular, any $P\in \ker \ell$ is of finite order $1,2$ or $3$. 
\item The group $\ker \ell$ contains $[G,G]$ and $\ker \ell=\boldsymbol{G_{\textsf{\bf per}}}=\{f\in G : \exists m\in \N^*, f^m=\id  \}$.
\item For any $\beta\in [0,1)$, the map $[R_\beta,g]$ has order $3$. \  \footnote{This will be used for maps in $G$, that is for $\beta=n\alpha \mod 1$.}
\end{enumerate}
\end{lemm}

\begin{proof} \ 

\begin{enumerate}
\item We obviously have $\ell(\id) = 0$ and $\ell (R_{\alpha})=1$. Given $f,h \in G$ and $x\in [0,1)$, the computation of $hf(x)$ leads to   
$$hf(x)= x +  \ (\ell (f) +\ell (h))\  \alpha \  + \frac{p_f(x)+p_h(f(x))}{3}$$
Therefore $$\ell(hf)=\ell (f) +\ell (h) \mbox{ \quad and \quad} p_{hf}(x)=p_f(x)+p_h(f(x))$$
In particular, $\ell$ is a morphism from $G$ to $(\Z,+)$.

\medskip

\item Let $f\in G$, one has $\ds \ell\left(f R_{\alpha} ^{-\ell(f)}\right) = \ell (f) +\ell (R_{\alpha} ^{-\ell(f)})=\ell (f) -\ell(f)=0$. Then \break
{$P_f =f \ R_{\alpha} ^{-\ell(f)} \in \ker \ell$ and $f= P _f \ R_{\alpha} ^{\ell(f)}$. \hfill\break Similarly, we get $Q_f= R_{\alpha} ^{-\ell(f)} \ f \in \ker \ell$ and $f= R_{\alpha} ^{\ell(f)} \ Q_f$.}

\medskip

\item Let  $x\in [0,1)$, the $\ker \ell$-orbit of $x$ is the set $\left\{f(x) \ | \ f\in \ker \ell  \right\}$ and it is included in $S= \left\{ x+ \frac{p}{3}, \ p \in \Z\right\} \cap [0,1)$ which has cardinality $3$ since $S\subset [0,1)$ and the distance between two distinct elements of $S$ is $\frac{1}{3}$ or  $\frac{2}{3}$.

\medskip

\item Since $\ell$ is a morphism to an abelian group, its kernel contains $[G,G]$. In addition, let $f\in G_{\textsf{per}}$ of order $m\in \N^*$ and $x\in [0,1)$.  One has $$f^m(x) = x+ \ m \ell(f) \ \alpha \ + \frac{p_{f^m}(x)}{3}= x$$ By irrationality of $\alpha$, we get $\ell(f)=0$ that is $f\in \ker \ell$.

\medskip

\item Let  $\beta \in [0, \frac{1}{3})$, we compute the map $[R_\beta,g]= R_{\beta} \ g \ R_{-\beta} \ g$.

$$\boldsymbol{[0, \beta)} \ra ^{\sct g}_{\sct +\frac{2}{3}} [\frac{2}{3}, \frac{2}{3}+\beta) \ra ^{R_{-\beta}}_{\sct -\beta} [\frac{2}{3}-\beta, \frac{2}{3}) \ra ^{\sct g}_{\sct +0}[\frac{2}{3}-\beta, \frac{2}{3}) \ra ^{R_{\beta}}_{\sct \beta} \boldsymbol{[\frac{2}{3}, \frac{2}{3}+\beta)}$$
$$\boldsymbol{[\beta, \frac{1}{3})} \ra ^{\sct g}_{\sct +\frac{2}{3}} [ \frac{2}{3}+\beta, 1) \ra ^{R_{-\beta}}_{\sct -\beta} [\frac{2}{3},1-\beta) \ra ^{\sct g}_{\sct -\frac{2}{3}}[0,\frac{1}{3} -\beta) \ra ^{R_{\beta}}_{\sct \beta} \boldsymbol{[\beta, \frac{1}{3})} \qquad \ $$
$$\boldsymbol{[\frac{1}{3}, \frac{1}{3}+\beta)} \ra ^{\sct g}_{\sct 0} [\frac{1}{3}, \frac{1}{3}+\beta) \ra ^{R_{-\beta}}_{\sct -\beta} [ \frac{1}{3}-\beta, \frac{1}{3}) \ra ^{\sct g}_{\sct +\frac{2}{3}}[1 -\beta, 1) \ra ^{R_{\beta}}_{\sct \beta-1} \boldsymbol{[0,\beta)} \ $$
$$\boldsymbol{[\frac{1}{3}+\beta, \frac{2}{3})} \ra ^{\sct g}_{\sct 0} [\frac{1}{3}+\beta, \frac{2}{3}) \ra ^{R_{-\beta}}_{\sct -\beta} [ \frac{1}{3}, \frac{2}{3}-\beta) \ra ^{\sct g}_{\sct 0} [ \frac{1}{3}, \frac{2}{3}-\beta)  \ra ^{R_{\beta}}_{\sct \beta} \boldsymbol{ [ \frac{1}{3}+\beta, \frac{2}{3}) } $$
$$\boldsymbol{[\frac{2}{3}, \frac{2}{3}+\beta)} \ra ^{\sct g}_{\sct -\frac{2}{3}} [0, \beta) \ra ^{R_{-\beta}}_{\sct 1-\beta} [1-\beta,1) \ra ^{\sct g}_{\sct -\frac{2}{3}}[\frac{1}{3} -\beta, \frac{1}{3}) \ra ^{R_{\beta}}_{\sct \beta} \boldsymbol{[\frac{1}{3},\frac{1}{3}+\beta)} \qquad \ $$
$$\boldsymbol{[\frac{2}{3}+\beta, 1)} \ra ^{\sct g}_{\sct -\frac{2}{3}} [\beta, \frac{1}{3}) \ra ^{R_{-\beta}}_{\sct -\beta} [ 0, \frac{1}{3}-\beta) \ra ^{\sct g}_{\sct +\frac{2}{3}} [ \frac{2}{3}, 1-\beta)  \ra ^{R_{\beta}}_{\sct \beta} \boldsymbol{[\frac{2}{3}+\beta, 1) }\qquad \ $$
\end{enumerate}

In conclusion, $[R_\beta,g]$ has support $[0, \beta) \sqcup [\frac{1}{3}, \frac{1}{3}+\beta) \sqcup [\frac{2}{3}, \frac{2}{3}+\beta)$ and its restriction is periodic of exact period $3$. Therefore, $[R_\beta,g]$ has order $3$.

\smallskip

The case $\beta \in (\frac{2}{3},1)$ can be deduced from the previous case by noting that $\beta'=1-\beta \in (0,\frac{1}{3})$ and $R_\beta = R_{\beta'}^{-1}$.

\smallskip

Finally, for the case $\beta \in (\frac{1}{3},\frac{2}{3})$, as $C=[R_\beta,g]$ has order $1$, $2$ or $3$, it suffices to check that the $C$-orbit of $0$ has cardinality more than $2$. This is provided by the computation of $C(0) =\frac{1}{3}$ and $C(\frac{1}{3}) =\frac{2}{3}$. \end{proof}

\begin{lemm}  \label{L4Nex} \ 
\begin{enumerate}
\item The group $\ker \ell$ is generated by $\{ R_\alpha ^ng  R_\alpha^{-n},  n \in \mathbb Z \}$.
\item The group $[G,G]$ is generated by $\{ [R_\alpha ^n,  g ],  n \in \mathbb Z \}$.
\item The group $[G,G]$ coincides with $[\ker \ell,\ker \ell]$.
\end{enumerate}
\end{lemm}

\begin{proof} \ 
\begin{enumerate}
\item It is plain that the maps $ R_\alpha ^ng  R_\alpha^{-n}, \ n\in \Z$ belong to $\ker \ell$. We previously noted that any $f\in G$ can be decomposed as 
$$f= R_\alpha^{p_s} g  R_\alpha^{p_{s-1}} g \cdots R_\alpha^{p_1}  g R_\alpha^{p_0},\mbox{ where } p_j \in \Z, $$
and then we can write: 
$$f= R_\alpha^{p_s} g  \boldsymbol{R_\alpha^{-p_s}}  \ \boldsymbol{R_\alpha}^{\boldsymbol{p_s} + p_{s-1}} g R_\alpha^{p_{s-2}}\cdots R_\alpha^{p_1}  g R_\alpha^{p_0}$$

Iterating this process, we get: 
$$f= R_\alpha^{p_s} g   R_\alpha^{-p_s}  \ \ R_\alpha^{p_s+ p_{s-1}} g R_\alpha^{-(p_s+ p_{s-1})}\ \cdots \ R_\alpha^{p_s+ p_{s-1}+ \cdots p_1}  g  R_\alpha^{-(p_s+ p_{s-1}+ \cdots p_1)}   \ \ R_\alpha^{ \sum_{j=0}^s p_j}$$

Moreover, $f\in \ker \ell$ if and only if $\sum_{j=0}^s p_j=0$, thus any $f\in \ker \ell$ is the product of elements of the form $R_\alpha^{p} g   R_\alpha^{-p}$. 

\medskip

\item We first note that for all $p,q \in \Z$, we have 
$$(*) \quad R_\alpha^{p} g   R_\alpha^{-p} \ \ R_\alpha^{q} g   R_\alpha^{-q} = R_\alpha^{p} g   R_\alpha^{-p} \boldsymbol{g}  \ \ \boldsymbol{g}  R_\alpha^{q} g   R_\alpha^{-q} =[R_\alpha^{p}, g ] \ \ [R_\alpha^{q}, g ]^{-1}.$$

Next, we claim that any element of $[G,G]$ can be written as a product of an even number of  $R_\alpha^{n} g R_\alpha^{-n}$   so as a product of $[R_\alpha^{p}, g ] $ and $[R_\alpha^{q}, g ]^{-1}$ by $(*)$.

Indeed, it suffices to prove this property for a commutator $f=[a,b]$ with $a,b\in G$. By Lemma \ref{L2Nex} (2), $f$ can be written as  $$f=P_a R_\alpha^{\ell(ab)} Q_b Q_a ^{-1}  R_\alpha^{-\ell(ab)}  P_b^{-1} = P_a \ R_\alpha^{\ell(ab)} Q_b \boldsymbol{R_\alpha^{-\ell(ab)}  \ R_\alpha^{\ell(ab)}} Q_a ^{-1}  R_\alpha^{-\ell(ab)} \  P_b^{-1}$$
According to Item $(1)$,  $P_a, P_b, Q_a$ and  $Q_b$ can be written as a product of $R_\alpha^{n} g  R_\alpha^{-n}$. Moreover,   for any $c \in G$, $Q_c = R_\alpha^{-\ell(c)} P_c  R_\alpha^{-\ell(c)}$ so the maps $P_c$, $P_c^{-1}$, $Q_c$, $Q_c^{-1}$ and their conjugates by $R_\alpha^{m}$ ($m\in \mathbb Z$) can be decomposed as products of the same number of $R_\alpha^{n} g   R_\alpha^{-n}$. Finally, $f$ can be written as a product of an even number of  $R_\alpha^{n} g   R_\alpha^{-n}$.

\medskip

\item By the previous point, showing that $[G,G] \subset [\ker \ell,\ker \ell]$ reduces to prove that $[R_\alpha^{n}, g ] \in  [\ker \ell,\ker \ell]$ for any $n\in \Z$ and we claim that $\ds \boldsymbol{[R_\alpha^{n}, g ]= [g, R_\alpha^{n} g R_\alpha^{-n}]}$. 

Indeed, noting that if $a^2= b^2= \id$ then $[a,b]= (ab) ^2$ and that $g$ and $R_\alpha^{n} g R_\alpha^{-n}$ have order 2, we get  $\ds [g, R_\alpha^{n} g R_\alpha^{-n}]= (g R_\alpha^{n} g R_\alpha^{-n}) ^2 = [g, R_\alpha^{n}]^2 =  [R_\alpha^{n},g]^{-2}$.

In addition, by Lemma \ref{L2Nex} (5), the map $\ds [R_\alpha^{n},g]=[R_{n\alpha},g]$ has order $3$ then $\ds [g, R_\alpha^{n} g R_\alpha^{-n}]= [R_\alpha^{n},g]$.

\end{enumerate}
\vskip -6mm
\end{proof}

\begin{lemm}  \label{L3Nex} \ 
\begin{enumerate}
\item Let $f\in \ker \ell$ and $J\subset [0,\frac{1}{3})$ such that $J$, $J + \frac{1}{3}$ and $J + \frac{2}{3}$  are continuity intervals of $f$ then there exists $\sigma=\sigma_{f,J} \in \mathfrak{S}_3$, the group of permutations of the set $\{0,1,2\}$ such that $f(J + \frac{i}{3}) = J + \frac{\sigma(i)}{3}$.
\item {Let $f,h\in \ker \ell$ and $J\subset [0,\frac{1}{3})$ such that $J$, $J + \frac{1}{3}$ and $J + \frac{2}{3}$  are continuity intervals of $f$ and $h$  then  $J$, $J + \frac{1}{3}$ and $J + \frac{2}{3}$  are continuity intervals of $fh$ and $$\sigma_{fh,J} = \sigma_{f,J} \circ \sigma_{h,J}$$}
\end{enumerate}
\end{lemm}

\begin{proof} \ 
\begin{enumerate}
\item Let $f\in \ker \ell$, since $J\subset [0,\frac{1}{3})$ is a continuity interval of $f$, the map $p_f$ is constant on $J$ equal to some $p\in \Z$ and $f(J) =\{ x+ \frac{p}{3}, x\in J\} = J + \frac{p}{3}$ is contained in $[\frac{p}{3},\frac{p+1}{3})\cap [0,1)$. Therefore $p\in \{0,1,2\}$ and $f(J)$ is either $J$, $J + \frac{1}{3}$ or $J + \frac{2}{3}$.

Analogously, for $i=1,2$, we get that $f(J+\frac{i}{3})$ is either $J$, $J + \frac{1}{3}$ or $J + \frac{2}{3}$. This defines the required $\sigma \in \mathfrak{S}_3$.

\medskip

\item Let $f,h \in \ker \ell$ and $J\subset [0,\frac{1}{3})$ such that  $f$ and $ h$ are continuous on $J$, $J + \frac{1}{3}$, $J + \frac{2}{3}$. For any $i\in\{0,1,2\}$, we have that $h$ is continuous on $J+\frac{i}{3}$ and 
$$fh(J+\frac{i}{3}) = f \left( J+ \frac{ \sigma_{h,J}(i)}{3} \right )=  J+ \frac{ \sigma_{f,J}\left(\sigma_{h,J}(i)\right)}{3}$$ since $f$ is continuous on the intervals $J+\frac{j}{3}$.
\end{enumerate}

\vskip -0.6cm \end{proof}

\noindent{\textbf{Proof of Proposition \ref{PNex}.}}

\medskip

By Item (1) of Corollary \ref{coro3}, the group G is not virtually nilpotent.

\medskip

We claim that all maps in $[G,G]\setminus \{\id\}$ commute and have order $3$.
\begin{quote} 
Indeed, let $C=[f,h] \in [G,G]\setminus \{\id\}$ be a non trivial commutator. By Lemma \ref{L4Nex} (3), we can suppose that $f,h \in \ker \ell$. Let us decompose $[0,1)$ as a union of triples of intervals $J$, $J + \frac{1}{3}$, $J + \frac{2}{3}$ with $J\subset [0, \frac{1}{3})$ such that $f, h , f^{-1}, h^{-1}$ are continuous on $J$, $J + \frac{1}{3}$, $J + \frac{2}{3}$. 

Let $J$ be such an interval, it follows from Item (2) of Lemma \ref{L3Nex} that $$\sigma_{[f,h],J}=\left[\sigma_{f,J} \ , \ \sigma_{h,J} \right] \in \left[ \mathfrak{S}_3, \mathfrak{S}_3\right] =\mathfrak{A}_3$$

Noting that the non trivial elements of $\mathfrak{A}_3$ are $3$-cycles, we get that $\sigma_{[f,h],J}$ is either trivial or has order $3$ and the same holds for $C_{|J}$ by its continuity. As $C\not=\id$, there is some $J$ such that $C_{|J}\not=\id$ and then $C$ has order $3$.

As $\mathfrak{A}_3$ is abelian, two commutators in $[G,G]$ have commuting associate permutations on their common continuity intervals, so they commute.

Finally, we get the claim since $[G,G]$ is generated by the commutators and the product of two elements of order $3$ that commute is either of order $3$ or trivial.
\end{quote}

\medskip

In addition, $[G,G]$ is not finitely generated. This is due to the fact that any finitely generated abelian torsion group has to be finite but the $[G,G]$-elements $[R_\alpha^n, g]$, $n\in \mathbb Z$ are pairwise distinct, since the proof of Item (5) of Lemma \ref{L2Nex} indicates that they have pairwise distinct break point sets. 

\subsection{Non virtually solvable examples.} \label{ExNoViSo} \

Let $n\in \N^*$, $n\geq 5$, we make the alternating group $\mathcal A_n$ act on $[0,1)$ as a subgroup of $G_n$ so that the \textbf{IET associated to $a\in \mathcal A_n$}, $t_a$, sends the interval $J_p=[\frac{p-1}{n}, \frac{p}{n})$ to $J_{a(p)}=[\frac{a(p)-1}{n}, \frac{a(p) }{n})$ and has a trivial rotation vector. In particular, for any $x\in [0,1)$ there exists an integer $p_a(x)$ such that $t_a(x) = x + \frac{p_a(x)}{n}$. Let $G$ be the subgroup of $\iet$ generated by $\mathcal A_n$ and $R_{\alpha}$, $\alpha \in [0, \frac{1}{n}) \setminus \Q$. 

\medskip

Note that any $f\in G$ can be written as 
$$f= R_\alpha^{p_s} a_s R_\alpha^{p_{s-1}}a_{s-1} \cdots R_\alpha^{p_1} a_1 R_\alpha^{p_0},\mbox{ where } a_j \in \mathcal A_n \mbox{  and } p_j \in \Z.$$

\begin{prop} \label{PNex5}
The group $G$ is elementary amenable, non virtually solvable and its commutator subgroup $[G,G]$ is a perfect but not simple locally finite \footnote{i.e its finitely generated subgroups are finite} torsion group.
\end{prop}

\begin{lemm} \label{L1Nex5} 
Let $f \in G$. Then there exists a unique $\ell (f) \in \mathbb Z$ such that for all $x\in [0,1)$ one has 
$$f(x) = x + \ell (f) \alpha + \frac{p_f(x)}{n} \mbox{ , where } p_f(x) \in \mathbb Z.$$
\end{lemm}

\begin{proof} As in Lemma \ref{L1Nex}, we argue by induction on the length $L_S(f)$ of $f\in G$ as a word in $S=\{\mathcal A_n,R_\alpha\}$. The change being in the case $f=g f_0$ of the induction step. Here, we have $f= a f_0$ with $a\in \mathcal A_n$ and the involved translation $\epsilon \frac{2}{3}$ is replaced by $\frac{p_a(x)}{n}$.
\end{proof}

\begin{lemm} \label{L2Nex5} \ 
\begin{enumerate}
\item The map $\ell : G \to \mathbb Z$ is a morphism.
\item Any $f\in G$ can be written as $f= P_f R_{\alpha} ^{\ell(f)} = R_{\alpha}^{\ell(f)} Q_f $ where $P_f,Q_f \in \ker \ell$.
\end{enumerate}
\end{lemm}

The proof is analogous to that of Lemma \ref{L4Nex}.

\medskip

\begin{lemm}  \label{L4Nex5} \ 
\begin{enumerate}
\item The group $\ker \ell$ is generated by $\{ R_\alpha ^k a  R_\alpha^{-k}, \ k \in \mathbb Z , \  a \in \mathcal A_n\}$.
\item The group $\ker \ell$ is perfect and $[G,G]=\ker \ell$.
\item The $\ker \ell$-orbit of $x\in [0, \frac{p}{n})$ is $\{ {\displaystyle x+\frac{p}{n},} \ \  p= 0, \cdots, n-1 \}$ and \ $\ker \ell=G_{\textsf{per}}$.
\end{enumerate}
\end{lemm}

\begin{proof} \ 
\begin{enumerate}

\item We previously noted that any $f\in G$ can be written as 
$$f= R_\alpha^{p_s} a_s  R_\alpha^{p_{s-1}} a_{s-1} \cdots R_\alpha^{p_1}  a_1 R_\alpha^{p_0},\mbox{ where } p_j \in \Z, $$
and as in the proof of Lemma \ref{L4Nex} we get: 
$$f= R_\alpha^{p_s} a_s  R_\alpha^{-p_s}  \ \ R_\alpha^{p_s+ p_{s-1}} a_{s-1} R_\alpha^{-(p_s+ p_{s-1})}\ \cdots \ R_\alpha^{p_s+ p_{s-1}+ \cdots p_1} a_1 R_\alpha^{-(p_s+ p_{s-1}+ \cdots p_1)} \ \ R_\alpha^{ \sum_{j=0}^s p_j}$$
Moreover, as $\ell$ is a morphism $f\in \ker \ell$ if and only if $\sum_{j=0}^s p_j=0$, thus any $f\in \ker \ell$ is the product of elements of the form $R_\alpha^{k} a R_\alpha^{-k}$. 

\medskip

\item Let $f\in \ker \ell$, by the previous point, $f$ can be written as $$f=\prod_{j} R_\alpha ^{n_j} a_j  R_\alpha^{-n_j} \mbox{ where  } n_j \in \mathbb Z \mbox{ and } a_j \in \mathcal A_n$$
As $\mathcal A_n$ is perfect and $\mathcal A_n < \ker \ell $, any $a_j$ belongs to $[\ker \ell, \ker \ell]$. Finally, $[\ker \ell, \ker \ell]$ being normal in $G$, any $R_\alpha ^{n_j} a_j  R_\alpha^{-n_j}$ belongs to $[\ker \ell, \ker \ell]$ and then $f$ belongs to $[\ker \ell, \ker \ell]$, which establishes that $\ker \ell$ is perfect.

In addition, as $\ell$ is a morphism to an abelian group, $[G,G]<\ker \ell$ and the other inclusion holds since $\ker \ell$ is perfect.

\medskip

\item Let $x\in [0, \frac{p}{n})$ and $f\in \ker \ell$. Lemma \ref{L1Nex5} gives $f(x)= x+\frac{p_f(x)}{n}$ for some $p_f(x)\in \Z$ and $f(x)\in[0,1)$ forces $p_f(x)\in \{0,1,\cdots, n-1\}$. Moreover, by choosing $a\in \mathcal A_n$ with $a(1)=p$, we obtain $t_a(x)= x+\frac{p-1}{n}$. Therefore the $\mathcal{A}_n$-orbit of $x$ is $ \bigl\{ {\displaystyle x+\frac{p}{n},}  \ \ p=0, \cdots, n-1 \bigr\}$. Finally, $\ker \ell$ is contained in $\iet_\Q$ so in $G_{\textsf{per}}$ and the other inclusion is a consequence of $\alpha\notin \Q$.
\end{enumerate}

\vskip -7mm \end{proof}

\begin{defi} 
Let $f\in \ker \ell$, $x\in [0, \frac{1}{n})$ and $\beta \in (0, \frac{1}{n})$ such that $f$ is continuous on the intervals $J_p(x, \beta) = [x+\frac{p-1}{n}, x+\frac{p-1}{n}+\beta)$ for all  $p=1, \cdots, n$.

As the translations of $f$ belong to $\{\frac{p}{n},\ p\in \Z\}$, the map $f$ permutes the $J_p(x, \beta)$ and then there exists $\omega\in \mathcal S_n$ such that $$f(J_p(x,\beta))= J_{\omega (p)}(x,\beta)= [x+\frac{\omega (p)-1}{n}, x+\frac{\omega (p)-1}{n}+\beta)$$
It is easy to see that $\omega$ does not depend on the choice of suitable $\beta$ and it is denoted by $\omega (f,x)$ and called \textbf{the local permutation of $f$ at $x$}.
\end{defi}

\begin{prop} \label{WMorph} Let $x\in [0, \frac{1}{n})$.

The map $\omega_x : \left\{ \begin{array}{ll}
\ker \ell &\to \ \ \mathcal S_n \cr f &\mapsto \ \ \omega (f,x) \end{array} \right.$ is a morphism and its image is $\mathcal A_n$.
\end{prop}

\begin{proof} Let $x\in [0, \frac{1}{n})$, it is obvious that $\omega(\id,x)= \id$. Let $f,h \in \ker \ell$ and $\beta \in (0, \frac{1}{n})$ small enough so that $f$ and $h$ are continuous on the intervals $J_p(x, \beta)$. 
Hence, for any $p\in\{1, \cdots, n\}$, we have $$h\circ f (J_p(x, \beta)) = h\bigl(J_{\omega(f,x)(p)}(x, \beta) \bigr) = J_{\omega(h,x)\omega(f,x)(p)}(x, \beta),$$ therefore $\omega (h\circ f,x)= \omega(h,x)\ \omega(f,x)$.

\medskip

Since $\ker \ell$ is perfect the image of $\omega_x$ is contained in $[\mathcal S_n,\mathcal S_n]= \mathcal A_n$. In addition, it is easy to check that the IET $t$ associated to $\tau \in \mathcal A_n$ satisfies $\omega (t,x)= \tau$ and then  we get $\omega_x(\ker \ell)= \mathcal A_n$.
\end{proof}

\begin{prop}\label{Sbeta} For any $\beta \in (0, \frac{1}{n})$, the set $S_\beta$ of maps in $\ker \ell$ having support in $\displaystyle \bigcup _{p=1}^{n} \ \bigl[\frac{p-1}{n},\frac{p-1}{n}+\beta \bigr)$ is a normal subgroup of $\ker \ell$ and its image by $\omega_0$ is $\mathcal A_n$.
\end{prop}

\noindent \textit{Proof.} The normality of $S_\beta$ is direct consequence of the facts that the translation set of $\ker \ell$ is $\{\frac{p}{n},\  p\in \Z \cap [-(n-1), n-1] \}$ and the support of $h \circ f \circ h^{-1}$ is the image by $h$ of the support of $f$. Consequently, the  image by $\omega_0$ of $S_\beta$  is a normal subgroup of $\mathcal A_n$ which is simple, therefore $\omega_0(S_\beta)$ is either trivial or equal to $\mathcal A_n$. Hence, the proof of Proposition \ref{Sbeta} reduces to prove that $\omega_0(S_\beta)$ is not trivial. This is provided, taking $\beta \in (0, \frac{1}{n})\cap \{ \frac{p}{n} + k\alpha, \  p,k \in \Z \}$ so that $R_\beta \in G$, by the following

\begin{lemm} \label{com} Let $\beta \in (0, \frac{1}{n})$, $\tau \in \mathcal A_n$ and $t\in\ker \ell$ be its associated IET.

Then the map $C_{\beta,t} = [R_\beta,t]= R_{\beta} t R_{-\beta} t^{-1}$ has support in $\displaystyle \bigcup _{p=1}^{n} \ \bigl[\frac{p-1}{n}, \frac{p-1}{n}+\beta \bigl)$ and its local permutation at $0$ is $\omega (C_{\beta,t},0)= [\sigma, \tau]$, where $\sigma$ is the $n$-cycle $(1,2,\cdots,n) \in \mathcal S_n$.
\end{lemm}

\begin{proof}
We explicitely compute the map $C_{\beta,t} = [R_\beta,t]= R_{\beta} t R_{-\beta} t^{-1}$. We first note that the break point set of  $C_{\beta,t}$ is contained in $\{\frac{p-1}{n}, \ 
\frac{p-1}{n}+\beta ; \ p \in \{1, ..., n\} \}$ and we have 
$$\ \hskip -6mm \boldsymbol{\bigl[\frac{p-1}{n}, \frac{p-1}{n}+\beta\bigr)} 
\ra ^{\sct t^{-1}} \bigl[\frac{\tau^{-1}(p)-1}{n}, \frac{\tau^{-1}(p)-1}{n}+\beta \bigr) 
\ra ^{R_{-\beta}} \bigl[\frac{\tau^{-1}(p)-1}{n} -\beta, \frac{\tau^{-1}(p)-1}{n}\bigr) $$
$$\ \hskip -4mm \ra ^{\sct t} \bigl[\frac{\tau(\tau^{-1}(p)-1)}{n} -\beta, \frac{\tau(\tau^{-1}(p)-1)}{n}\bigr)  
\ra ^{R_{\beta}} \boldsymbol{\bigl[\frac{\tau(\tau^{-1}(p)-1)}{n}, \frac{\tau(\tau^{-1}(p)-1)}{n} +\beta \bigr)}$$
and 
$$\boldsymbol{      \bigl[\frac{p-1}{n}+\beta, \frac{p}{n} \bigr)} 
\ra ^{\sct t^{-1}}  \bigl[\frac{\tau^{-1}(p)-1}{n}+\beta, \frac{\tau^{-1}(p)}{n} \bigr)
\ra ^{R_{-\beta}}   \bigl[\frac{\tau^{-1}(p)-1}{n} , \frac{\tau^{-1}(p)}{n} -\beta \bigr) $$
$$\ra ^{\sct t}     \bigl[\frac{p-1}{n}, \frac{p}{n}-\beta \bigr) \ra ^{R_{\beta}}   
 \boldsymbol{       \bigl[\frac{p-1}{n}+\beta, \frac{p}{n} \bigr) }$$
This proves that $C_{\beta,t}$ has support in $\displaystyle \bigcup _{p=1}^{n} \ \bigl[\frac{p-1}{n}, \frac{p-1}{n}+\beta \bigr)$. In addition, its local permutation at $0$ is given by $\omega (p)-1=\tau(\tau^{-1}(p)-1)$ that is
$$\omega (p)=\tau(\tau^{-1}(p)-1)+ 1= \sigma \tau \sigma^{-1} \tau ^{-1} = [\sigma, \tau].$$ 

\vskip -8mm \end{proof}

Note that, if $\tau \notin \CR(\sigma)$ then $R_\beta$ and $t$ are not commuting and $C_{\beta,t}$ is not trivial.

\medskip

\noindent \textbf{Proof of Proposition \ref{PNex5}.}

We first prove the properties related to $[G,G]$. Lemma \ref{L4Nex5} implies that $[G,G] =\ker \ell$ is a  perfect torsion group and, by Proposition \ref{Sbeta}, it is not simple because of its non trivial normal subgroups $S_\beta$.

We claim that the finitely generated subgroups of $[G,G] =\ker \ell$ are finite.
\begin{quote}
Indeed, let $H=\langle m_i \rangle < \ker \ell$ be a  finitely generated subgroup of $\ker \ell$. By Properties \ref{proBP}, the set $\BP(H)$ is contained in the $H$-orbit of the finite set $\cup \BP(m_i)$ so it is finite since all $\ker \ell$-orbits are finite. 
\end{quote} 
The group $G$ is elementary amenable since $G/[G,G] \simeq \Z$ is abelian and $[G,G] =\ker \ell$ is the direct union of its finitely generated subgroups that are finite.

\smallskip

By Item (1) of Corollary \ref{coro3}, $G$ is not virtually nilpotent. Moreover, we are going to prove that no finite index subgroup of $\ker \ell$ is solvable. Indeed, let $R$ be a finite index subgroup of $\ker \ell$. Since a subgroup of finite index always contains a normal subgroup  also of finite index, we may suppose that $R$ is normal in $G$. Therefore for any $x\in [0, \frac{1}{n})$, the group $\omega_x(R)$ is either trivial or equal to  $\mathcal A_n$. As $\ker \ell$ is infinite, $R$ is not trivial and so there exists $x\in [0, \frac{1}{n})$ such that $\omega_x(R)$ is not trivial and then equal to $\mathcal A_n$. 

Finally, since $\mathcal A_n$ is simple, $\omega_x$ maps any iterated commutators subgroup of $R$ surjectively to $\mathcal A_n$. In particular, $R$ is not solvable and therefore $\ker \ell$ and then $G$ are not virtually solvable.

\smallskip

\begin{rema} Due to the SAF-invariant (see \cite{Bos}), the groups constructed in Sections \ref{ExMeta} and \ref{ExNoViSo}, with rationally independent $\alpha$ are not conjugated in $\iet$. Moreover, similar arguments show that Proposition \ref{PNex5} holds for the groups generated by $R_{\alpha_1}, ..., R_{\alpha_m}$ and $\mathcal A_n$ for $n\geq 5$. 
In particular, these groups are not linear since by Schur Theorem (\cite{Sch}) any linear torsion group is virtually abelian. 
\end{rema}

\bibliographystyle{alpha}
\bibliography{RefIET}

\begin{thebibliography}{JMBMdlS18}

\bibitem[Arn81]{Ar}
Pierre Arnoux.
\newblock \'echanges d'intervalles et flots sur les surfaces.
\newblock In {\em Ergodic theory ({S}em., {L}es {P}lans-sur-{B}ex, 1980)
  ({F}rench)}, volume~29 of {\em Monograph. Enseign. Math.}, pages 5--38. Univ.
  Gen\`eve, Geneva, 1981.

\bibitem[Bos16]{Bos}
Michael Boshernitzan.
\newblock Subgroup of interval exchanges generated by torsion elements and
  rotations.
\newblock {\em Proc. Amer. Math. Soc.}, 144(6):2565--2573, 2016.

\bibitem[BS62]{BS}
Gilbert Baumslag and Donald Solitar.
\newblock Some two-generator one-relator non-{H}opfian groups.
\newblock {\em Bull. Amer. Math. Soc.}, 68:199--201, 1962.

\bibitem[Cor20]{Cor}
Yves Cornulier.
\newblock Realizations of groups of piecewise continuous transformations of the
  circle.
\newblock {\em Journal of Modern Dynamics}, 16(0):59--80, 2020.

\bibitem[DFG20]{DFG}
Fran\c{c}ois Dahmani, Koji Fujiwara, and Vincent Guirardel.
\newblock Solvable groups of interval exchange transformations.
\newblock {\em Ann. Fac. Sci. Toulouse Math. (6)}, 29(3):595--618, 2020.

\bibitem[GL19a]{GL}
Nancy Guelman and Isabelle Liousse.
\newblock Distortion in groups of affine interval exchange transformations.
\newblock {\em Groups Geom. Dyn.}, 13(3):795--819, 2019.

\bibitem[GL19b]{GLrev}
Nancy Guelman and Isabelle Liousse.
\newblock Reversible maps and products of involutions in groups of iets.
\newblock {\em arXiv:1907.01808}, 2019.

\bibitem[Gro81]{Gro}
Mikhael Gromov.
\newblock Groups of polynomial growth and expanding maps.
\newblock {\em Inst. Hautes \'Etudes Sci. Publ. Math.}, (53):53--73, 1981.

\bibitem[JMBMdlS18]{JMBMS}
Kate Juschenko, Nicol\'{a}s Matte~Bon, Nicolas Monod, and Mikael de~la Salle.
\newblock Extensive amenability and an application to interval exchanges.
\newblock {\em Ergodic Theory Dynam. Systems}, 38(1):195--219, 2018.

\bibitem[Kea75]{Ke}
Michael Keane.
\newblock Interval exchange transformations.
\newblock {\em Math. Z.}, 141:25--31, 1975.

\bibitem[May43]{May}
A.~Mayer.
\newblock Trajectories on the closed orientable surfaces.
\newblock {\em Rec. Math. [Mat. Sbornik] N.S.}, 12(54):71--84, 1943.

\bibitem[Min97]{Mi}
Hiroyuki Minakawa.
\newblock Classification of exotic circles of {${\rm PL}_+(S^1)$}.
\newblock {\em Hokkaido Math. J.}, 26(3):685--697, 1997.

\bibitem[Nav11]{NaB}
Andr\'{e}s Navas.
\newblock {\em Groups of circle diffeomorphisms}.
\newblock Chicago Lectures in Mathematics. University of Chicago Press,
  Chicago, IL, spanish edition, 2011.

\bibitem[Nov09]{No}
Christopher~F. Novak.
\newblock Discontinuity-growth of interval-exchange maps.
\newblock {\em J. Mod. Dyn.}, 3(3):379--405, 2009.

\bibitem[OS15]{FS}
Anthony~G. O'Farrell and Ian Short.
\newblock {\em Reversibility in dynamics and group theory}, volume 416 of {\em
  London Mathematical Society Lecture Note Series}.
\newblock Cambridge University Press, Cambridge, 2015.

\bibitem[Sch11]{Sch}
Issai Schur.
\newblock Uber gruppen periodischer substitutionen.
\newblock {\em Sitzungsb. Preuss. Akad. Wiss.}, page 619–627, 1911.

\end{thebibliography}
\end{document}